\documentclass[a4paper,12pt]{amsart}
\usepackage{amsmath}
\usepackage{array}
\usepackage{amsthm}
\usepackage{amssymb}
\usepackage[mathscr]{eucal}

\usepackage[all]{xy}

\theoremstyle{plain}
\newtheorem*{thm*}{Theorem}
\newtheorem*{prop*}{Proposition}
\newtheorem*{rem*}{Remark}

\newtheorem{thm}{Theorem}[section]

\newtheorem{defi}[thm]{Definition}
\newtheorem{prop}[thm]{Proposition}
\newtheorem{lm}[thm]{Lemma}
\newtheorem{obs}[thm]{Observation}
\newtheorem{obs*}{Observation}

\newtheorem{claim*}{Claim}
\newtheorem{exple}[thm]{Example}
\newtheorem{nota}[thm]{Notation}

\numberwithin{equation}{thm}

\theoremstyle{remark}
\newtheorem{rem}[thm]{Remark}

\newcommand{\B}{{\mathcal B}}

\newcommand{\N}{\mathbb N}

\newcommand{\kk}{\Bbbk} 
\newcommand{\kkm}{\Bbbk \text{-Mod}} 
\newcommand{\kkmg}{\Bbbk \text{-grMod}} 

\newcommand{\Sh}{\Gamma_{sh}}

\newcommand{\C}{{\mathcal{C}}}
\newcommand{\D}{{\mathcal{D}}}

\newcommand{\LL}{{\mathcal{L}}}

\newcommand{\PP}{{\mathcal{P}}}

\newcommand{\GLie}{{\Sh^{Lie}}}

\usepackage{multicol}

\makeatletter
 {\end{multicols}\if@restonecol\onecolumn\else\clearpage\fi}
\makeatother
	\newdir{>>}{{}*!/3.5pt/:(1,-.2)@^{>}*!/3.5pt/:(1,+.2)@_{>}*!/7pt/:(1,-.2)@^{>}*!/7pt/:(1,+.2)@_{>}}
\newdir{ >>}{{}*!/8pt/@{|}*!/3.5pt/:(1,-.2)@^{>}*!/3.5pt/:(1,+.2)@_{>}}
\newdir{ |>}{{}*!/-3.5pt/@{|}*!/-8pt/:(1,-.2)@^{>}*!/-8pt/:(1,+.2)@_{>}}
\newdir{ >}{{}*!/-8pt/@{>}}
\newdir{>}{{}*:(1,-.2)@^{>}*:(1,+.2)@_{>}}
\newdir{<}{{}*:(1,+.2)@^{<}*:(1,-.2)@_{<}}

\makeindex

\begin{document}

	 \title{Leibniz homology of Lie algebras as functor homology}

		      \author{Eric Hoffbeck}
		      \author{Christine Vespa}
\address{Universit\'e Paris 13, Sorbonne Paris Cit\'e, LAGA, CNRS (UMR 7539)
99 avenue Jean-Baptiste Cl\'ement,
93430 Villetaneuse,
France}
\email{hoffbeck@math.univ-paris13.fr}

\address{ Universit\'e de Strasbourg, Institut de Recherche
 Math\'ematique Avanc\'ee, Strasbourg, France. }                        
\email{vespa@math.unistra.fr}

\date{\today}

\begin{abstract}
We prove that Leibniz homology of Lie algebras can be described as functor homology in the category of linear functors from a category associated to the Lie operad.\\
\vspace{.3cm}
\textit{Mathematics Subject Classification: 17B55, 18D50, 17B56} 
\vspace{.3cm}

\textit{Keywords}: Functor homology, operads, Lie algebras, Leibniz homology
\end{abstract}

\maketitle

\begin{minipage}{\textwidth}

\printindex
\end{minipage}

\tableofcontents

\section*{Introduction}

In the 90's, Robinson and Whitehouse defined and studied in \cite{RW2002}
$\Gamma$-homology of commutative algebras, a homology theory
suited in the context of differential graded modules over a field of positive characteristic.
Pirashvili and Richter proved in \cite{PiraR-2000} that this homology theory can be interpreted as functor homology, for functors from the category $\Gamma$ of finite pointed sets. This category can be viewed as a category associated to the commutative set operad.
For associative algebras, similar results are obtained by Pirashvili and Richter in \cite{PiraR-2002}.
In this paper, the authors interpret usual Hochschild and cyclic homology of associative algebras as functor homology. In this setting, the category $\Gamma$ is replaced by a category associated to the associative set operad.
In \cite{Liv-Richter}, Livernet and Richter give a description of $E_n$-homology of non-unital commutative algebras as functor homology. In this setting the category $\Gamma$ is replaced by a suitable category of epimorphisms related to planar trees with $n$-levels.

In all these results, homology theories are obtained as Tor functors in a category of functors, between a Loday functor and a functor $t$ playing the role of the base ring. In \cite{Pira-Hodge}, Pirashvili uses the interpretation of homology theories as functor homology to give purely homological proofs of Hodge decompositions of higher order Hochschild homology of commutative algebras.

This paper is motivated by the following natural question:
is it possible to describe homologies of Lie algebras as functor homology? 

We obtain that the
Leibniz homology of Lie algebras can be interpreted as functor homology. More precisely, we prove the following theorem:

\begin{thm*}
For $A$ a Lie algebra and $M$ a $A$-module, one has an isomorphism:
$$H_*^{Leib}(A,M)\simeq Tor^{\GLie}_*(t,\LL^{Lie}_{sh}(A,M))$$
where $\GLie$ is a suitable linear category associated to the Lie operad and $\LL^{Lie}_{sh}(A,M)$ is a generalized Loday functor.
\end{thm*}

\vspace{1cm}
Leibniz homology was defined by Loday in \cite{Loday1} in order to understand periodicity phenomena in algebraic K-theory (see \cite{Loday2}). This homology can be defined via an explicit complex, 
obtained as the non-commutative analogue of the Chevalley-Eilenberg complex and it is defined for Leibniz algebras, which are non-commutative variants of Lie algebras. In Leibniz algebras, the bracket is not required
to be anti-commutative anymore, and  the Jacobi relation is replaced by the Leibniz relation,
which can be seen as a lift of the Jacobi relation in the non-commutative context.
As Lie algebras are a particular case of Leibniz algebras, the Leibniz homology can be used to
compute homological invariants of Lie algebras.

The most striking result concerning Leibniz homology is the following Loday-Cuvier theorem \cite{Cuvier, Loday-Cyclic}: for an associative algebra $A$, the Leibniz homology of the Lie algebra $gl(A)$ is isomorphic to
the free associative algebra over the Hochschild homology of $A$. 
This theorem is the non-commutative variant of the Loday-Quillen-Tsygan theorem \cite{LQ, FT}
which states that the Chevalley-Eilenberg homology of $gl(A)$ is isomorphic to
the free exterior algebra over the cyclic homology of $A$. Recall that cyclic homology is naturally isomorphic to the additive K-theory which is the analogue of K-theory obtained by replacing the general linear group $GL(A)$ by the Lie algebra $gl(A)$.

\vspace{1cm}

In order to write Leibniz homology of Lie algebras as functor homology, two objects have to be defined: 
a category playing for Lie algebras the role of $\Gamma$ and a variant of the Loday functor for this category.
One main difference with the cases of commutative and associative algebras is that the operad encoding Lie algebras is not a set operad.
This requires the use of linear categories (i.e. categories enriched over the category of $\kk$-modules), instead of usual categories.

The proof of the main theorem is based on a characterization of the Tor functors. The principal difficulty is to prove that the homology vanishes on projective generators. This result is the heart of this paper. To prove it, we need to consider a linear category associated to the operad $Lie$, called $\GLie$, 
with a shuffle condition, requiring some maps to preserve a part of the order. This condition is the main reason why we obtain Leibniz homology as functor homology and not Chevalley-Eilenberg homology.

The proof of the vanishing of the Leibniz homology on projective generators can be decomposed into three steps. We  begin to describe a basis of the morphisms spaces in the linear category $\GLie$. This allows us to define a filtration on the $\kk$-module associated to the complex computing the homology of projective generators. The proof of the compatibility of this filtration with the differential requires to use a basis of the operad $Lie$ and to understand its behaviour with respect to composition. Then we identify the associated graded complex with a sum of acyclic complexes. This step is heavily based on the thorough understanding of the combinatorial objects associated to the basis of the morphisms spaces in the linear category $\GLie$.

\vspace{1cm}
In a recent work \cite{Fresse}, Fresse proves that operadic homology can be obtained as functor homology, by a method different than ours. 
For the commutative operad and the associative operad, Fresse recovers results of  Pirashvili and Richter \cite{PiraR-2000, PiraR-2002}. 
For the operad $Lie$, Fresse obtains a description of Chevalley-Eilenberg homology of Lie algebras as functor homology over the category $\Gamma^{Lie}$. Our theorem shows that functor homology over the category $\GLie$ is naturally related to Leibniz homology of Lie algebras. The category $\GLie$ has the advantage to be more manageable than the category $\Gamma^{Lie}$ for effective computations.

\vspace{1cm}
The paper is organized as follows: 
Section 1 consists in recollections on functor homology of enriched categories and on 
some operadic definitions.
In Section 2 we define our category $\GLie$ and the associated Loday functor.
After some recollections of Leibniz homology, we state in Section 3 the main theorem of the paper.
The last section is devoted to the proof of the vanishing of the Leibniz homology on projective generators.

\section*{Acknowledgements}
The first author would like to thank Birgit Richter for explaining to him the usual strategy to interpret homology theories as functor homology. The authors gratefully acknowledge the hospitality of the Isaac Newton Institut in Cambridge. 
The first author is supported in part by the Sorbonne-Paris-Cit\'e IDEX grant Focal and the ANR grant ANR-13-BS02-0005-02 CATHRE. The second author is supported by the ANR grant ANR-11-BS01-0002 HOGT.

\vspace{2cm}
\textbf{Notations}: The following categories will be useful throughout the whole paper.
\begin{itemize}
\item $Set$ is the category of sets with morphisms the set maps;
\item $\kkm$ is the category of modules over a fixed commutative ground ring $\kk$;
\item $\kkmg$ is the category of $\N$-graded modules over $\kk$;
\item $\Delta$ is the simplicial category, i.e. the category with objects ordered finite sets $[n]$ and morphisms order preserving maps;
\item $\Gamma$ is the skeleton of the category of finite pointed sets having as objects $[n]=\{0,1, \ldots, n\}$ with $0$ as basepoint and morphisms the set maps $f:[n] \to [m]$ such that $f(0)=0$;
\item $\Gamma^{surj}$ is the category having as objects finite pointed sets and as morphisms the pointed surjective maps;
\item $\Sh$ is  the category with objects ordered finite sets $[n]$ and morphisms pointed shuffling maps $f$, that is maps such that $min(f^{-1}(i))<min(f^{-1}(j))$ whenever $i<j$;
\item $\Sh^{surj}$ is  the category with objects finite ordered sets $[n]$ and morphisms pointed shuffling surjections;
\end{itemize}

The Lie algebras we consider in this paper are algebras over the Lie operad. In particular, in characteristic 2, a Lie bracket is just antisymmetric and we do not have in general $[x,x]=0$.


\vspace{2cm}
\section{Recollections on enriched category, functor homology and operads}
In this section we briefly recall some definitions and facts about enriched categories, functor homology and operads useful in the sequel.

\subsection{Enriched  category}
One of the standard references for symmetric monoidal categories and enriched categories is the book of Borceux \cite[Chapter 6]{Borceux2}.
Let $(\C, \otimes, 1)$ be a closed symmetric monoidal category. Recall that a symmetric monoidal category $\C$ is closed when for each object $C \in \C$, the functor $- \otimes C: \C \to \C$ admits a right adjoint denoted by $[C,-]: \C \to \C$.
\begin{defi}
 A category $\D$ enriched  over $\C$ (or a $\C$-category) consists of a class $I$ (representing the objects of $\D$) and for any objects $i,j,k \in I$ an object of $\C$: $\D(i,j)$ (representing the morphisms from $i$ to $j$ in $\D$) and morphisms in $\C$
$$\D(i,j) \otimes \D(j,k) \to \D(i,k) \quad \text{and} \quad 1\to \D(i,i)$$
(representing the composition of morphisms in $\D$ and the identity morphism on $i$). These structure morphisms are required to be associative and unital in the obvious sense.
\end{defi}

\begin{exple}
\begin{enumerate}
\item A category enriched over $(Set, \times, [0])$ is an usual category.
\item A category enriched over $(\kkm, \otimes, \kk)$ is a $\kk$-linear category and  a functor enriched over $(\kkm, \otimes, \kk)$ is a $\kk$-linear functor.
\item The category $\C$ can be provided with the structure of an enriched category over $(\C, \otimes, 1)$ using the bifunctor $[-,-]: \C^{op} \times \C \to \C$ whose composition with the forgetful functor $\C(1,-): \C \to Set$ is just $\C(-,-): \C^{op} \times \C \to Set$.

\end{enumerate}
\end{exple}

\begin{defi}
Let $\B$ and $\D$ categories enriched over $\C$. A functor enriched over $\C$ (or a $\C$-functor) from $\B$ to $\D$, $F:\B \to \D$ consists of an object of $\D$, $F(B)$ for every object $B \in \B$ and of morphisms in $\C$:
$$F(B,B'): \B(B,B') \to \D(F(B), F(B'))$$
for every pair of objects $B,B' \in \B$ that are associative and unital. 
\end{defi}

\begin{defi}
A natural transformation enriched over $\C$, $\sigma: F \to F'$, between two functors enriched over $\C$, $F,F': \B \to \D$ consists in giving, for every object $B \in \B$, a morphism:
$$\sigma_B: 1 \to \D(F(B), F'(B))$$
in $\C$, 
for every object $B$ of $\B$, that satisfy obvious commutativity conditions for a natural transformation.
\end{defi}

We denote by $\C\text{-Nat}(F,F')$ the object of natural transformations enriched over $\C$ between $F$ and $F'$.

\begin{thm}[\textbf{Enriched Yoneda lemma}]
Let $\B$ be a small category enriched over $\C$. For every object $B \in \B$ and every functor enriched over $\C$, $F: \B \to \C$, the object of natural transformations enriched over $\C$ from $\B(B,-)$ to $F$ exists and there is an isomorphism in $\C$:
$$\C\text{-Nat}(\B(B,-),F) \simeq F(B)$$
which is natural both in $F$ and in $B$.
\end{thm}

\subsection{Functor homology}

\begin{defi}
For $\D$ a $\C$-category and $F: \D^{op} \to \C$ and $G: \D \to \C$ a pair of $\C$-functors, the enriched tensor product of $F$ and $G$, $F \underset{\D}{\otimes}G$ is the coequalizer
$$\xymatrix{
\underset{d,d' \in D}{\coprod} F(d') \otimes \D(d,d') \otimes G(d) \ar@<+0.4ex>[r]^-{U} \ar@<-0.4ex>[r]_-{B}&\underset{d \in D}{\coprod} F(d) \otimes G(d) \ar[r] & F \underset{\D}{\otimes}G
}$$
in $\C$ where the map $U$ is induced by the composite
$$F(d') \otimes \D(d,d') \otimes G(d) \xrightarrow{1 \otimes F(d,d')} F(d') \otimes \C(F(d'), F(d)) \otimes G(d) \xrightarrow{ev \otimes 1} F(d) \otimes G(d) $$
where $F(d,d')$ is the morphism in $\C$ given because $F$ is a $\C$-functor and $ev$ is the evaluation map. The map $B$ is induced by a similar composite  with $G$ in place of $F$.
\end{defi}
Recall that the evaluation map is the counit of the adjunction on $\C$ of the monoidal product with the internal-hom.

For a category $\D$ enriched over $\C$, we call left $\D$-modules covariant $\C$-functors from $\D$ to  $\C$ and right $\D$-modules contravariant $\C$-functors from $\D$ to $\C$. Let $\D$-mod (resp. mod-$\D$ ) the category of left (resp. right) $\D$-modules.
If $\C$ is an abelian category, the categories $\D$-mod and mod-$\D$ are abelian. In the sequel $\C$ is an abelian category.

\begin{prop} \label{exact}
The bifunctor $-\underset{\D}{\otimes}-: \text{mod-}\D \times \D\text{-mod} \to \C$ is right exact with respect to each variable.
\end{prop}

We have the following characterization of homology theories in the enriched setting:

\begin{prop} \label{char-homo}
Let $\D$ be a $\C$-category and $G$ a right $\D$-module.
 If $H_*$ is a functor from $\D$-mod to $\C$ such that 
\begin{enumerate}
 \item $H_*$ sends short exact sequences to long exact sequences,
 \item for all $F \in \D$-mod we have a natural isomorphism $H_0(F) \simeq G\otimes_\D F$ ,
 \item $H_i(F)=0$ for all projective $\D$-modules $F$ and $i>0$,
\end{enumerate}
then, for all $F \in \D$-mod, we have a natural isomorphism 
$$H_i(F)\simeq Tor_i^\D(G,F).$$
\end{prop}

In the rest of the paper, all our categories will be enriched over  $(\kkm, \otimes_\kk, \kk)$.


\subsection{Algebraic operads}
A good reference for operads is the book of Loday and Vallette \cite{LV}.
We denote by $OrdSet$ the category of finite ordered sets (with order-preserving bijections as morphisms)
and by $Fin$ the category of finite sets (with bijections as morphisms).

\begin{defi}
A collection (resp. symmetric collection) is a contravariant functor from the category $OrdSet$ 
(resp. $Fin$) to the category $\kkm$. 
\end{defi}

For a collection $P$, we denote $P([n-1])$ by $P(n)$ for $n \geq 0$.

Let $P$ and $Q$ be two symmetric collections. One defines their symmetric composition $P \circ Q$ by 

$$(P \circ Q)(n) = \bigoplus_{m} P(m) \otimes_{\Sigma_m} \left( \bigoplus_{\alpha \in Set([n-1],[m-1]) } Q(\alpha^{-1}(0)) \otimes \ldots  \otimes Q(\alpha^{-1}(m-1)) \right).$$

Let $P$ and $Q$ be two collections.

One defines their nonsymmetric composition $P \circ Q$ by 
$$(P \circ Q)(n) = \bigoplus_{m} P(m) \otimes \left( \bigoplus_{\alpha \in \Delta([n-1],[m-1])} Q(\alpha^{-1}(0)) \otimes \ldots  \otimes Q(\alpha^{-1}(m-1)) \right).$$

When $P(0)=Q(0)=0$, one defines their shuffle composition $P \circ_{sh} Q$ by 
$$(P \circ_{sh} Q)(n) = \bigoplus_{m} P(m) \otimes \left( \bigoplus_{\alpha \in \Sh([n-1],[m-1])}Q(\alpha^{-1}(0)) \otimes \ldots  \otimes Q(\alpha^{-1}(m-1)) \right).$$

In the rest of the paper, unless otherwise stated, all our collections will be reduced, that is $P(0)=0$. 
This hypothesis will be required to get a functor from symmetric operads to shuffle operads.

\begin{rem}
Note that 
$$(P \circ_{sh} Q)(n) \simeq \bigoplus_{m} P(m) \otimes \left( \bigoplus_{\alpha \in \Sh^{surj}([n-1],[m-1])}Q(\alpha^{-1}(0)) \otimes \ldots  \otimes Q(\alpha^{-1}(m-1)) \right)$$
since $Q(0)=0$.
\end{rem}

\begin{rem}
Note that the previous definitions and remark can be extended to a symmetric monoidal category $(\C, \otimes, 1)$ having an initial object $0$ satisfying $C \otimes 0 \simeq 0 \simeq 0 \otimes C$, $\forall C \in \C$. This is the case for $(Set, \times, [0])$, the initial object being $\emptyset$.
\end{rem}

\begin{defi}
A linear nonsymmetric (resp. shuffle) operad is a monoid $P$ in the category of collections equipped with the nonsymmetric (resp. shuffle) composition.
A linear symmetric operad is a monoid $P$ in the category of symmetric collections equipped with the symmetric composition.
\end{defi}

An algebra $A$ over an operad $P$ is a $\kk$-module endowed with an action of $P$, that is there are maps 
$\theta : P(n) \otimes A^{\otimes n} \to A$ compatible with the monoid structure of $P$
(and with the symmetric group action in the symmetric case).

For $P$ a symmetric operad, an $A$-module $M$ is a $\kk$-module endowed with $M \oplus A$ has a $P$-algebra structure extending the 
structure on $A$ and satisfying
$$\theta (P(n) \otimes (M^{\otimes k} \otimes A^{\otimes n-k}))=0 \textrm { if } n \geq k \geq 2.$$

\begin{exple} \label{uCom}
\begin{enumerate}

\item 
Usual commutative algebras can be seen as algebras over a symmetric operad called $Com$, 
determined by $Com(n)=\kk$ for all $n\geq1$ and $Com(0)=0$, and equipped with the obvious composition.
The notion of $A$-modules for $A$ a commutative algebra is the usual one.
The case of unital commutative algebras can be dealt using the symmetric operad called $uCom$, 
determined by $uCom(n)=\kk$ for all $n\geq 0$ and equipped with the obvious composition.

\item Lie algebras can be seen as algebras over a symmetric operad called $Lie$.
This operad is generated by a symmetric bracket which satisfies the Jacobi relation.
The notion of $A$-modules for $A$ a Lie algebra is the usual one.

\end{enumerate}
\end{exple}

There exists a forgetful functor $(-)^{sh}$ from symmetric collections to nonsymmetric collections.
This functor allows us to make a direct link between the notion of symmetric operads and the notion of shuffle operads because of the following lemma,
proved by Dotsenko and Khoroshkin:
\begin{lm} \cite[Proposition 3]{DK2010} 
The forgetful functor $(-)^{sh}$ is monoidal.
\end{lm}

Therefore this functor induces a forgetful functor from symmetric operads to shuffle operads.

\begin{exple} 
The usual $Lie$ symmetric operad induces a linear shuffle operad. 
\end{exple}

The fundamental property we will use later in the paper is the following: 
For $\PP$ a symmetric operad, the $\kk$-modules $\PP(n)$ and $\PP^{sh}(n)$ are the same. For instance they have the same bases.
This observation also implies that for a $\PP$-algebra $A$, the maps $\theta : P(n) \otimes A^{\otimes n} \to A$ defining 
the algebra structure can be seen as maps $P^{sh}(n) \otimes A^{\otimes n} \to A$. The same holds for maps
defining a $A$-module structure on $M$.


\vspace{2cm}
\section{Functorial constructions associated to an operad}
\subsection{Enriched category associated to an operad}

The notion of enriched category associated to an operad is quite classical, already appearing in a paper of May and Thomason \cite{MayT}.
We recall the classical definition and then extend it to the shuffle context.

\begin{defi}
\begin{itemize}
\item The linear category of pointed operator $\Gamma^\PP$ associated to a reduced \textbf{symmetric} linear operad $\PP$ is the linear category whose objects are pointed finite sets $[n]$ and such that:
$$\Gamma^\PP([n],[m])=\underset{\alpha \in \Gamma([n],[m])}{\bigoplus}\PP( \alpha^{-1}(0) ) \otimes \ldots \otimes \PP( \alpha^{-1}(m)).$$
Composition of morphisms is prescribed by symmetric operad structure maps in $\PP$.
\item The linear category of operator $\Sh^{\PP}$ associated to a reduced \textbf{shuffle} linear operad $\PP$  is the linear category whose objects are pointed  finite sets $[n]$ and such that:
$$\Sh^{\PP}([n],[m])=\underset{\alpha \in \Sh([n],[m])}{\bigoplus}\PP( \alpha^{-1}(0)) \otimes \ldots \otimes \PP( \alpha^{-1}(m)).$$
Composition of morphisms is prescribed by shuffle operad structure maps in $\PP$.
\end{itemize}
\end{defi}

\begin{rem}
Note that for a symmetric operad $\PP$ satisfying $\PP(0)=0$, we have
$$\Gamma^\PP([n],[m])\simeq \underset{\alpha \in \Gamma^{surj}([n],[m])}{\bigoplus}\PP( \alpha^{-1}(0) ) \otimes \ldots \otimes \PP( \alpha^{-1}(m))$$
and
$$\Sh^{\PP}([n],[m])\simeq \underset{\alpha \in \Sh^{surj}([n],[m])}{\bigoplus}\PP( \alpha^{-1}(0)) \otimes \ldots \otimes \PP( \alpha^{-1}(m)).$$

\end{rem}
For a shuffle operad $\PP^{sh}$ coming from a symmetric operad $\PP$, we ease notation by abbreviating $\Sh^{\PP^{sh}}$ to $\Sh^{\PP}$.

Note that $\Sh^\PP([n],[0])= \PP^{sh}([n])= \PP([n])=\PP(n+1)$.

\begin{exple}
We explain graphically how the composition works in $\Gamma^\PP_{sh}$ for $\PP$ a shuffle operad. 

Let $(\alpha,f)$ be the generator of $\Gamma^\PP_{sh}([9],[5])$ where $\alpha \in \Gamma_{sh}([9],[5])$ is given by:
$$\alpha(0)=\alpha(4)=0; \ \alpha(1)=1;\ \alpha(2)=\alpha(3)=\alpha(7)=2;\ $$
$$\alpha(5)=3;\ \alpha(6)=\alpha(9)=4;\ \alpha(8)=5$$
and 
$$f=f_0 \otimes f_1 \otimes f_2 \otimes f_3 \otimes f_4 \otimes f_5$$
 where $\forall i \in\{0,1,2,3,4,5\}\ f_i \in \PP(\alpha^{-1}(i))$. We can represent this morphism by the following picture:
$$ \vcenter
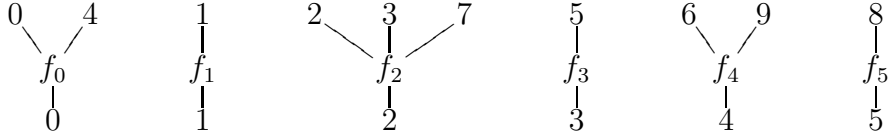
{\xymatrix@M=1pt@H=0pt@R=10pt@C=14pt@!0{
 0 \ar@{-}[ddr] & & 4 \ar@{-}[ddl]& & & 1\ar@{-}[dd] &  & & 2\ar@{-}[ddrr] & & 3\ar@{-}[dd] & & 7\ar@{-}[ddll] & & & 5\ar@{-}[dd] & & & 6\ar@{-}[ddr] & & 9\ar@{-}[ddl] & & & 8 \ar@{-}[dd] \\
 &\\
 & f_0\ar@{-}[dd] & & & & f_1 \ar@{-}[dd]& & & & & f_2\ar@{-}[dd] & & & & & f_3\ar@{-}[dd] & & & & f_4 \ar@{-}[dd]& & & & f_5 \ar@{-}[dd]& \\
 &\\
 &0&&&&1&&&&&2&&&&&3&&&&4&&&&5\\
}} $$
Let $g$ be the generator of $\Gamma^\PP_{sh}([5],[2])$ represented by the following picture:
$$ \vcenter
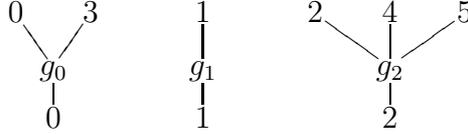
{\xymatrix@M=1pt@H=0pt@R=10pt@C=14pt@!0{
 0 \ar@{-}[ddr] & & 3 \ar@{-}[ddl]& & & 1\ar@{-}[dd] &  & & 2\ar@{-}[ddrr] & & 4\ar@{-}[dd] & & 5\ar@{-}[ddll] \\
 &\\
 & g_0\ar@{-}[dd] & & & & g_1 \ar@{-}[dd]& & & & & g_2\ar@{-}[dd] \\
 &\\
 &0&&&&1&&&&&2\\
}} $$

To compose them, we use the composition of the shuffling maps $\alpha \in \Gamma_{sh}$ and the shuffle composition of the operad $\PP$.

We obtain the following element in $\Gamma^\PP_{sh}([9],[2])$, where the dotted lines shows where the shuffle composition of $\PP$ has to be done.

$$ \vcenter
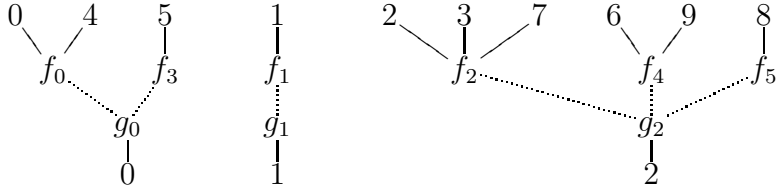
{\xymatrix@M=1pt@H=0pt@R=10pt@C=14pt@!0{
0\ar@{-}[ddr] && 4\ar@{-}[ddl] && 5\ar@{-}[dd] &&& 1\ar@{-}[dd] &&& 2\ar@{-}[ddrr] && 3\ar@{-}[dd] && 7\ar@{-}[ddll] && 6\ar@{-}[ddr] && 9\ar@{-}[ddl] && 8\ar@{-}[dd] \\
&\\
& f_0 \ar@{.}[ddrr] &&& f_3\ar@{.}[ddl] &&& f_1 \ar@{.}[dd]&&&&& f_2\ar@{.}[ddrrrrr] &&&&& f_4\ar@{.}[dd] &&& f_5\ar@{.}[ddlll] \\
&\\
&&& g_0\ar@{-}[dd] &&&& g_1\ar@{-}[dd] &&&&& &&&&& g_2\ar@{-}[dd] \\
&\\
 &&&0&&&&1&&&&&&&&&&2&&&&&&&\\
}}$$

This element can also be written as

$$ \vcenter
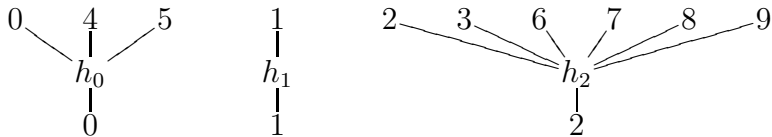
{\xymatrix@M=1pt@H=0pt@R=10pt@C=14pt@!0{
0\ar@{-}[ddrr] && 4\ar@{-}[dd] && 5\ar@{-}[ddll] &&& 1\ar@{-}[dd] &&& 
2\ar@{-}[ddrrrrr] && 3\ar@{-}[ddrrr] && 6\ar@{-}[ddr] && 7\ar@{-}[ddl] && 8\ar@{-}[ddlll] && 9 \ar@{-}[ddlllll]  \\
& \\
&& h_0\ar@{-}[dd] &&&&& h_1\ar@{-}[dd] &&&&& &&& h_2\ar@{-}[dd] \\
& \\
 &&0&&&&&1&&&&&&&&2&&&&&&&&&\\
}}$$
where the element $h_2 \in \PP (\{2,3,6,7,8,9\})$ is given by the composition $g_2 \circ (f_2, f_4, f_5)$ 
associated to the shuffling map $\alpha: \{2,3,6,7,8,9\} \to \{0,1,2\}$ defined by $\alpha^{-1}(0)=\{2,3,7\}$, $\alpha^{-1}(1)=\{6,9\}$ and 
$\alpha^{-1}(2)=\{8\}$. The elements $h_0$ and $h_1$ are defined in a similar way.

\end{exple}

For $\PP=Lie$, there are some particular maps in $\Gamma^{Lie}$ and $\Sh^{Lie}$ which are central in our later constructions. 

\begin{defi} \label{dij}
For every $n \geq0$ and for $0\leq i <j \leq n+1 $, 
there is one element in  $\Sh^{Lie}([n+1],[n])$ called $d_{i,j}$ associated to the single map $f \in\Sh ([n+1],[n])$ sending $i$ and $j$ to $i$,
where the element in $Lie(1)$ is the identity and the element $\mu$ in $Lie(2)$ is the $Lie$ bracket.
$$ \vcenter{\xymatrix@M=2pt@H=0pt@R=16pt@C=14pt@!0{
 0\ar@{-}[dd]& & 1\ar@{-}[dd] & & \ldots\ar@{-}[dd] & & i\ar@{-}[dr] & & j\ar@{-}[dl] 
& & i+1\ar@{-}[dd] & & \ldots\ar@{-}[dd] & & \hat j&  & \ldots\ar@{-}[dd] & & n+1\ar@{-}[dd] \\
& &  & &  & & & \mu\ar@{-}[d] & & &  & &   & & &  &  & &  \\
0& & 1 & & \ldots & & & i & & & i+1 & & \ldots  & & &  & \ldots & & n  }} $$
\end{defi}

\begin{rem}
These elements $d_{i,j}$ are actually generators (via linear combination and composition) of $\bigoplus_{n,m}\GLie([n],[m])$.
\end{rem}

\begin{rem} \label{linear}
If $S$ is a symmetric set operad, we consider the corresponding linear symmetric operad $\kk S$ defined as $\kk S(n):=\kk[S(n)]$ with $\kk[-]$ denoting the free $\kk$-module. The $\kk$-linearization of a category $\C$ , which will be denoted by $\kk[\C]$, is the $\kk$-linear category which has the same objects as $\C$ and whose sets of morphisms $\kk[\C](c,c')$ are the $\kk$-modules $\kk[\C(c,c')]$ generated by the sets of morphisms of $\C$.

It is easy to check that $\Gamma^{\kk S} = \kk [\Gamma(S)]$, where $\Gamma(S)$ denotes the category associated to $S$.

For example, $\Gamma^{uCom}=\kk[\Gamma]$ and $\Gamma^{Com}=\kk[\Gamma^{surj}]$, where $Com$ and $uCom$ are the symmetric linear operad defined in Example \ref{uCom}.

Furthermore, for $F: \C \to \kkm$ and $G: \D^{op} \to \kkm$  and $\kk[F]: \kk[\C] \to \kkm$ and $\kk[G]: \kk[\D^{op}] \to \kkm$ their linearizations, we have:
$$Tor^{\C}(F,G) \simeq Tor^{\kk[\C]}(\kk[F], \kk[G]).$$
\end{rem}

\begin{obs} \label{projectifs}
Let $\PP$ be a symmetric operad.
By the enriched Yoneda lemma, the representable functors $\Gamma^\PP([n],-)$ for $n \geq 0$ are projective generators of the category of $\kk$-linear functors from $\Gamma^\PP$ to $\kkm$. 
Similarly, the representable functors $P_n= \Gamma^\PP_{sh}([n],-)$ for $n \geq 0$ are projective generators of the category of linear functors from $\Gamma^\PP_{sh}$ to $\kkm$.

Notice that for $G$ a $\kk$-linear functor from $(\Gamma^\PP)^{op}$ to $\kkm$ and $F$ a $\kk$-linear functor from $\Gamma^\PP$ to $\kkm$, we have:
$$G \underset{\Gamma^\PP}{\otimes} \Gamma^\PP([n],-) \simeq G([n]); \qquad \Gamma^\PP(-,[n]) \underset{\Gamma^\PP}{\otimes} F \simeq F([n]).$$
\end{obs}

\subsection{Loday functor}

We recall the definition of the Loday functor, appearing in \cite{Loday-Cyclic}.
This functor appears in various contexts (associative or commutative algebras, pointed or unpointed sets)

Let $A$ be a unital commutative algebra and $M$ an $A$-module.
The Loday functor $\LL(A,M): \Gamma \to \kkm$ is defined by 
$\LL(A,M)([n])= M \otimes A^{\otimes n}$
 and for $f$ an element of $\Gamma([n],[m])$ 
the induced map $f_*:\LL(A,M)([n]) \to \LL(A,M)([m])$ is given by
$$f_*(a_0 \otimes a_1 \otimes \ldots \otimes a_n)=b_0 \otimes \ldots \otimes  b_m$$
where $b_i= \underset{j \in f^{-1}(i)}{\prod} a_j$ or $b_i=1$ if $f^{-1}(i)=\emptyset$.

We first extend this definition for algebras over a symmetric operad, 
and then see how it can be adapted in the shuffle context.

Let $\PP$ be a symmetric operad, $A$ be a $\PP$-algebra and $M$ a $A$-module. Let us call $\theta$ the structure maps.

\begin{lm}
There exists a linear functor:
$$\LL^\PP(A,M):\Gamma^\PP \to \kkm$$
defined on an object $[n]$ by: 
$$\mathcal{L}^\PP(A,M)([n])=M \otimes A^{\otimes n}$$
and on a generator $(\alpha,f)$ of $\Gamma^\PP([n],[m])$ where $\alpha: [n] \to [m]$ and $f=f_{0} \otimes \ldots \otimes f_m$ (with $f_i \in \PP(\alpha^{-1}(i))$), by: 
$$\mathcal{L}^\PP(A,M)(\alpha,f)(a_0 \otimes a_1 \otimes \ldots \otimes a_n)=b_0 \otimes \ldots \otimes  b_m$$
where $b_i=\theta ( f_i \otimes \underset{j \in \alpha^{-1}(i)}{\bigotimes} a_j)$.

We extend the definition on a morphism in $\Gamma^\PP([n],[m])$ by linearity.
\end{lm}

\begin{proof}
We need to prove first that $\mathcal{L}^\PP(A,M)(\alpha,f)$ is well-defined, and then show the functoriality.

First, note that $b_0$ is in $M$ because $a_0$ is in $M$, the other $a_j$'s are in $A$ and because $\alpha(0)=0$. Moreover, the linearity is obtained by construction.

Concerning the functoriality, recall that the composition in $\Gamma^\PP$ is defined by the operadic composition.
Recall also that the structure maps $\theta$ (for the module $M$ and the algebra $A$) are required to be compatible with the operadic composition.
This requirement gives exactly $\mathcal{L}^\PP(A,M)(id,1_\PP \otimes \ldots \otimes 1_\PP)=id$ and $\mathcal{L}^\PP(A,M)((\alpha,f)(\beta,g)) = \mathcal{L}^\PP(A,M)(\alpha,f)\mathcal{L}^\PP(A,M)(\beta,g)$.
\end{proof}

\begin{defi}
The functor $\LL^\PP(A,M): \Gamma^\PP \to \kkm$ is called the symmetric Loday functor 
associated to the $\PP$-algebra $A$ and the $A$-module $M$.
\end{defi}

\begin{rem}
For $A$ a unital commutative algebra and $M$ an $A$-module, we have $\LL^{uCom}(A,M)=\Bbbk [\LL(A,M)]$ where $uCom$ is the symmetric operad defined in Example \ref{uCom}. 
\end{rem}
\begin{rem}
For $A$ a commutative algebra and $M$ a $A$-module, by \cite{Pira-Hodge} and \cite{PiraR-2000} we have
$$H_*^{\Gamma}(A,M) \simeq Tor^{\Gamma}(t, \LL(A,M))$$
where $H_*^{\Gamma}(A,M)$ is the $\Gamma$-homology of $A$ with coefficients in $M$ and $t: \Gamma^{op} \to Ab$ is the cokernel of $\Gamma(-, [2]) \to \Gamma(-, [1])$. By the previous remark and Remark \ref{linear} this result can be rephrased in linear categories as:
$$H_*^{\Gamma}(A,M) \simeq Tor^{\Gamma^{uCom}}(\kk[t], \LL^{uCom}(A,M)).$$
\end{rem}

\bigskip

Let $\PP$ be a symmetric operad, $A$ be a $\PP$-algebra and $M$ a $A$-module. 
Recall that $\PP^{sh}$ is the shuffle operad associated to $\PP$.
Let us call $\theta$ the structure maps.

\begin{lm}
There exists a linear functor:
$$\LL^\PP_{sh}(A,M):\Gamma^\PP_{sh} \to \kkm$$
defined on an object $[n]$ by: 
$$\mathcal{L}^\PP_{sh}(A,M)([n])=M \otimes A^{\otimes n}$$
and on a generator $(\alpha,f)$ of $\Gamma^\PP_{sh}([n],[m])$ where $\alpha: [n] \to [m]$ and $f=f_{0} \otimes \ldots \otimes f_m$ (with $f_i \in \PP_{sh}(\alpha^{-1}(i))$), by: 
$$\mathcal{L}^\PP_{sh}(A,M)(\alpha,f)(a_0 \otimes a_1 \otimes \ldots \otimes a_n)=b_0 \otimes \ldots \otimes  b_m$$
where $b_i=\theta ( f_i \otimes \underset{j \in \alpha^{-1}(i)}{\bigotimes} a_j)$.
We extend the definition on a morphism in $\Gamma^\PP_{sh}([n],[m])$ by linearity.
\end{lm}

\begin{proof}
The proof is mostly the same as in the symmetric case. The only difference comes from the shuffle context.
Recall that the structures maps in the symmetric context induce structures maps in the shuffle context.
These induced maps are compatible with the shuffle operadic composition.
\end{proof}

\begin{defi}
The functor $\LL^\PP_{sh}(A,M): \Gamma^\PP_{sh} \to \kkm$ is called the shuffle Loday functor
 associated to the $\PP$-algebra $A$ and the $A$-module $M$.
\end{defi}


\vspace{2cm}
\section{Leibniz homology of a functor}
Leibniz homology of a Lie algebra was defined by Loday.
In this section, we first recall the usual Leibniz complex of a Lie algebra, and then generalize it to define the Leibniz homology of a $\GLie$-module. In fact, Leibniz homology of $\LL^{Lie}_{sh}(A,M)$ is Leibniz homology of the Lie algebra $A$ with coefficients in $M$. Then we state the main result of this paper.

\subsection{Leibniz complex}
Let $A$ be a Lie algebra and $M$ a $A$-module.

\begin{defi} \cite{Loday-Cyclic}
The Leibniz homology of $A$ with coefficients in $M$, denoted by $HL_*(A,M)$, is the homology of the complex
$(C^{Leib}_* (A,M),d)$ where 
$C^{Leib}_n (A,M)= M \otimes A^{\otimes n}$ and the differential $d: C^{Leib}_n (A,M) \to C^{Leib}_{n-1} (A,M)$ is given by:
$$d(x \otimes a_1 \otimes \ldots \otimes a_n)= \sum_{1\leq i<j \leq n } (-1)^j x\otimes a_1 \otimes \ldots a_{i-1} \otimes [a_i,a_j]\otimes \ldots \otimes \widehat{a_j} \otimes \ldots \otimes a_n$$
$$+ \sum_{1 \leq j \leq n} (-1)^j [x,a_j] \otimes a_1 \otimes \ldots \otimes \widehat{a_j} \otimes \ldots \otimes a_n.$$
\end{defi}

\begin{rem}
\begin{itemize}
\item
If we replace $x$ by the symbol $a_0$, we can notice that:
$$d(a_0 \otimes a_1 \otimes \ldots \otimes a_n)= \sum_{0\leq i<j \leq n} (-1)^j a_0\otimes a_1 \otimes \ldots a_{i-1} \otimes [a_i,a_j]\otimes \ldots \otimes \widehat{a_j} \otimes \ldots \otimes a_n$$
Note that $d^2=0$ is proved in the same way as for the Chevalley-Eilenberg complex. 
\item Leibniz homology and Chevalley-Eilenberg homology are related by a map $H^{Leib}_* (A,M) \to H^{CE}_* (A,M)$ induced by the anti-symmetrisation map $C^{Leib}_* (A,M) \to C^{CE}_* (A,M)$. 
\end{itemize}
\end{rem}

\begin{prop}
Let $T$ be  a linear functor from $\GLie$ to $\kkm$. The map $d: T([n]) \to T([n-1])$ defined by $d=T(\sum_{0\leq i < j \leq n} (-1)^j d_{i,j})$ satisfies $d^2=0$. Therefore the sequence:
$$\ldots \xrightarrow{d}T([n])\xrightarrow{d} T([n-1])  \xrightarrow{d} \ldots$$
is a well-defined complex, denoted by  $(C^{Leib}_* (T),d)$.
\end{prop}

\begin{proof}
The relation $d^2=0$ is obtained by exactly the same computation as $d^2=0$ for the Leibniz complex of an algebra.
\end{proof}

\begin{defi}
Let $T$ be  a linear functor from $\GLie$ to $\kkm$.
The Leibniz homology of $T$, denoted by $H^{Leib(T)}_*$, is the homology of the complex $(C^{Leib}_* (T),d)$. Leibniz homology produces a functor from left $\GLie$-modules to $\kkmg$ called $H^{Leib}_*$.
\end{defi}

\begin{rem}
If one chooses the Loday functor for $T$, the definition coincides with the previous one: $(C_*^{Leib} (\LL^{Lie}_{sh}(A,M)),d)=(C_*^{Leib}(A,M),d)$.
\end{rem}


\subsection{Leibniz homology as functor homology}

Before stating our theorem, we first need to define a right $\GLie$-module, which will serve as basepoint.

\begin{defi}
The contravariant linear functor $t$ from $\GLie$ to $\kkm$ is defined by $\text{coker}\big(\GLie(-,[1]) \xrightarrow{(d_{0,1})_*} \GLie(-,[0])\big)$,
where $-_*$ denotes the postcomposition and $d_{0,1} \in \GLie([1],[0])$ is defined in Definition \ref{dij}.
\end{defi}

It is easy to check that $(\text{Im } (d_{0,1})_*)[n] =\GLie([n],[0])$ for $n\geq 1$ and  $(\text{Im } (d_{0,1})_*)[0]=0$.
Therefore $t([n])=0$ for $n\geq 1$ and $t[0]=\GLie([0],[0])=\kk$.

We can now state the main theorem of the paper:

\begin{thm}
Let $T$ be a left $\GLie$-module. Then 
$$H^{Leib}_*(T) = Tor_* ^\GLie (t, T).$$
\end{thm}

To prove this theorem, we use the characterization of homology theories given in \ref{char-homo},
which relies on three hypotheses.

The first hypothesis ($H^{Leib}_*$ sending short exact sequences to long exact sequences) is satisfied,
as $H^{Leib}_*$ is defined via the homology of a complex.

The functor $t$ has been defined in a way to satisfy the second hypothesis (concerning $H_0$). 

\begin{prop}
For all left $\GLie$-modules $T$, we have 
$$H_0^{Leib}(T)=t \otimes_{\Sh^{Lie}} T.$$
\end{prop}

\begin{proof}
We follow the ideas of the proof of Theorem 1.3 of \cite{PiraR-2002}. 
By the definition of $t$, we have an exact sequence:
$$\GLie(-,[1]) \xrightarrow{(d_{0,1})_*} \GLie(-,[0]) \to t \to 0.$$
For $T$ a left $\GLie$-module, by Proposition \ref{exact} the functor $- \underset{\GLie}{\otimes}T$ is right exact, so we deduce the following exact sequence:
$$\GLie(-,[1]) \underset{\GLie}{\otimes}T \to \GLie(-,[0]) \underset{\GLie}{\otimes}T \to t\underset{\GLie}{\otimes}T \to 0.$$
By Observation \ref{projectifs}, $\GLie(-,[1]) \underset{\GLie}{\otimes}T \simeq T([1])$ and $\GLie(-,[0]) \underset{\GLie}{\otimes}T \simeq T([0])$ so:
$$t\underset{\GLie}{\otimes}T  \simeq T([0])/Im\big(T([1]) \to T([0])\big)=H_0^{Leib}(T).$$

\end{proof}

The third hypothesis consists in the vanishing of the homology of the projective generators.
The proof of this part is detailed in the remaining section of the paper.

\vspace{2cm}
\section{Vanishing of the homology of the projective generators}

The goal of this section is to prove the following proposition.
\begin{prop} \label{Leib-projectifs}
For $i>0$ and $n \geq 0$,
$$H_i^{Leib}(P_n)=0$$
where $P_n=\GLie([n],-)$.
\end{prop}

We can note that for $n=0$, we directly obtain $H^{Leib}_i(P_0)=0$, for $i>0$, because $P_0([m])=0$ if $m>0$.

From now on, we fix a postive integer $n$.

Let us first give a survey of the three steps of the proof of this proposition. In the first step, we give a description of bases of the Lie operad and of $\GLie([n],[m])$. In Lemma \ref{bases}, we identify these bases with 
easy combinatorial objects (tuples, or ordered partitions into tuples). 
One important point is to consider all possible values of $m$ at the same time.
This allows us to use the ordering between tuples to obtain, in Proposition \ref{filtration}, a filtration of the $\kk$-module $\oplus_m C_m^{Leib}(P_n)$. Then we show, in Proposition \ref{compatible} that this filtration is compatible with the differential of the complex $C_*^{Leib}(P_n)$. This step in based on the crucial lemma \ref{crochet} 
describing the bracket of two elements of the basis of $Lie$ in terms of elements of the basis.
In the last part of the proof we identify the associated graded complex with a sum of acyclic complexes. This step is heavily based on the thorough understanding of relationships between tuples and ordered partitions into tuples. We deduce Proposition \ref{Leib-projectifs} by a classical spectral sequence argument.

\subsection{Filtration as a $\kk$-module}

\subsubsection{Basis of the Lie operad}

Recall that one can define the free symmetric operad generated by a symmetric collection, and that the notion of operadic ideal exists (cf \cite{LV}).
Therefore symmetric operads can be defined by generators and relations. 

Our operad of interest in this paper, $Lie$, can then be defined by the free symmetric operad generated by the symmetric collection
$\kk \mu$ concentrated in arity $2$, where $\mu$ denotes an antisymmetric bracket, quotiented by the ideal generated by 
the Jacobi relation. 

Recall that the Jacobi relation can be written operadically as
\begin{equation} \label{Jacobi}
\begin{array}{c}
$\xymatrix@M=1pt@C=6pt@R=6pt{
1 \ar@{-}[ddrr] & & 2\ar@{-}[dr] & & 3 \ar@{-}[dl] \\
 & & & *{\mu}\ar@{-}[dl]  & \\
 & & *{\mu}\ar@{-}[d]  & & \\
& &  & & 
}$\end{array}
=
\begin{array}{c}
 $\xymatrix@M=1pt@C=6pt@R=6pt{
1 \ar@{-}[dr] & & 2\ar@{-}[dl] & & 3 \ar@{-}[ddll] \\
 & *{\mu}\ar@{-}[dr]  & &  & \\
 & & *{\mu}\ar@{-}[d]  & &\\
& &  & &  
}$\end{array} 
-
\begin{array}{c}
 $\xymatrix@M=1pt@C=6pt@R=6pt{
1 \ar@{-}[dr] & & 3\ar@{-}[dl] & & 2 \ar@{-}[ddll] \\
 & *{\mu}\ar@{-}[dr]  & &  & \\
 & & *{\mu}\ar@{-}[d] & & \\
& &  & & 
}$\end{array}. 
\end{equation}

In the rest of the paper, all internal vertices of the trees are labelled with the bracket $\mu$, 
therefore we do not write it anymore.

Methods have been developed in \cite{Eric} and \cite{DK2010} to find monomial bases with good properties with respect to an ordering. 

A basis of the $Lie$ operad is given by the following planar trees (cf \cite{Eric}):

$$
\left\{ \vcenter{ \xymatrix@M=0pt@H=0pt@R=13pt@C=13pt@!0{
1\ar@{-}[dr]& &\sigma(2)\ar@{-}[dl] \\
& \ar@{-}[dr] & & \sigma(3)\ar@{-}[dl]  \\
& & \ar@{-}[dr] & & \cdots \ar@{-}[dl] \\
& & & \ar@{-}[dr] & & \cdots \ar@{-}[dl] \\
& & & & \ar@{-}[dr] & & \sigma(n) \ar@{-}[dl] \\
& & & & & \ar@{-}[d] \\
& & & & & &  
} } \  \Bigg\vert \ n \geq 1, \sigma \text{ permutation of } \{2, \ldots, n\} \right\}. $$

Note that this operadic basis corresponds to a basis of the multilinear part of the free Lie algebra on $x_1, \ldots, x_n$ 
given by elements of the form $[ \ldots[[x_1, x_{\sigma(2)}], x_{\sigma(3)}], \ldots, x_{\sigma(n)}]$.
(See \cite[Section 5.6.2]{Reutenauer}).

We denote the part of the basis in arity $n$ by $\mathcal B^{Lie}(n)$. We define similarly $\mathcal B^{Lie}(I)$ for $I$ a finite ordered set of cardinality $n$, 
using the order preserving bijection between $I$ and $\{1, \ldots, n\}$.

We denote the whole basis by $\mathcal B^{Lie}=\bigcup_n \mathcal B^{Lie}(n)$.

\subsubsection{Explicit description of the category associated to the Lie operad}

To describe the category $\GLie$, it is enough to describe the vector spaces of morphisms and to understand the composition of elements of the basis.
We first give a $\kk$-basis for $\GLie([n],[m])$. 
Recall that $\GLie([n],[m])=\underset{\alpha \in \Sh([n],[m])}{\bigoplus} Lie (\alpha^{-1}(0)) \otimes \ldots \otimes Lie (\alpha^{-1}(m))$.

Obviously, if $m>n$, the basis is empty.

If $m \leq n$, we can describe a basis $\mathcal B^{\GLie}(n,m)$ of $\GLie([n],[m])$ via shuffle maps and a tuple of elements of a Lie basis:

$$\mathcal B^{\GLie}(n,m)= \bigcup_{\alpha \in \Sh([n],[m])} \mathcal B^{Lie}(\alpha^{-1}(0)) \times \ldots \times \mathcal B^{Lie}(\alpha^{-1}(m)).$$

We also define $\mathcal B^{\GLie}(n,\N) = \bigcup_m \mathcal B^{\GLie}(n,m)$.

\subsubsection{Notations for elements of the bases}
To avoid drawing huge forests of trees, we first identify the elements of $\mathcal B^{\GLie}(n,\N)$ as some particular tuples with splittings.
We begin by defining these particular tuples and introduce some notations.

\begin{obs}
All trees appearing in $\mathcal B^{Lie}(n)$ are left combs, so they are characterized by their leaves.
An element of $\mathcal B^{Lie}(n)$ can be identified with a $n$-tuple $(1, \sigma(2), \ldots, \sigma(n))$ where $\sigma$ is a permutation of $\{2, \ldots, n\}$,
just by reading the inputs from left to right. So we have a bijection $\mathcal B^{Lie}(n) \simeq \mathfrak{S}_{n-1}$, where $\mathfrak{S}_{n-1}$ is the permutation group. In the sequel, we identify these two sets.
\end{obs}

\begin{exple}
Let $(\alpha,f)$ be the generator of $\GLie([6],[2])$ where $\alpha \in \Gamma_{sh}([6],[2])$ is given by:
$$\alpha(0)=\alpha(1)=\alpha(5)=0;\  \alpha(2)=1;\  \alpha(3)=\alpha(4)=\alpha(6)=2$$
and $f=f_0 \otimes f_1 \otimes f_2$ where:
$$f_0=Id \in \mathcal B^{Lie}(3) \simeq \mathfrak{S}_{2};\  f_1=Id \in \mathcal B^{Lie}(1) \simeq \mathfrak{S}_{0};$$
$$ f_2=\tau_{1,2} \in \mathcal B^{Lie}(3) \simeq \mathfrak{S}_{2}.$$
We can represent this morphism by the following picture:
$$ 
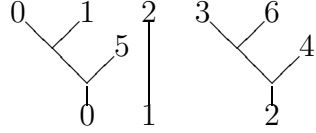
\vcenter{\xymatrix@M=0pt@H=0pt@R=13pt@C=13pt@!0{
0\ar@{-}[dr]& & 1\ar@{-}[dl] & \\
& \ar@{-}[dr]& & 5\ar@{-}[dl]\\
& & \ar@{-}[d] \\
& & 0&}}
\
\vcenter{\xymatrix@M=0pt@H=0pt@R=13pt@C=13pt@!0{
2\ar@{-}[ddd]\\
\\
\\
1&}}
\
\vcenter{\xymatrix@M=0pt@H=0pt@R=13pt@C=13pt@!0{
3\ar@{-}[dr]& & 6\ar@{-}[dl] & \\
& \ar@{-}[dr]& & 4\ar@{-}[dl]\\
& & \ar@{-}[d] \\
& & 2&}}$$

\end{exple}

\begin{defi} \label{STuple}
 A $m$-split $(n+1)$-tuple is an ordered partition of $[n]$ into $m+1$ tuples: 
$B_0 \cup B_1 \cup \ldots \cup B_m$
where $m \geq 0$ and $B_i=(k_{i,0}, \ldots, k_{i,\ell_i})$ satisfying two conditions:
\begin{enumerate}
 \item $0 \leq i<j \leq m \Rightarrow k_{i,0} < k_{j,0}$ ,
 \item $\forall \, 0 \leq i \leq m, \forall \,  0 \leq q \leq \ell_i, k_{i,0} < k_{i,q}$.
\end{enumerate}
To ease notations, we denote $B_0 \cup B_1 \cup \ldots \cup B_m$ by $(B_0 \mid B_1 \mid \ldots \mid B_m)$.
\end{defi}

\begin{rem}
The conditions impose that $k_{0,0}=0$. 
The first condition can be read as ``the first term of a tuple is smaller than the first term of the following tuples''. 
The second condition can be read as ``the first term of a tuple is smaller than the other terms of this tuple''. 
\end{rem}

\begin{nota}
We denote by $STuple(n,m)$ the set of $m$-split $(n+1)$-tuples, by $STuple(n,\N)$ the set $\cup_m STuple(n,m)$ 
and by $Tuple(n)$ the set $STuple(n,0)$ corresponding to the set of $(n+1)$-tuples starting with $0$.
\end{nota}

\begin{exple}
 For $n=2$, the set $STuple(2, \N)$ is composed of $6$ elements: 

$(0\mid1\mid2), (0\mid1,2), (0,1\mid2), (0,2\mid1), (0,1,2), (0,2,1).$

The set $Tuple(2)$ is composed of only $2$ elements: $(0,1,2)$ and $(0,2,1)$.

\end{exple}

\begin{lm} \label{condition3}
The two conditions $(1)$ and $(2)$ in Definition \ref{STuple} are jointly equivalent to the following single condition:
\begin{equation}
 \tag*{(3)} \forall \, 0 \leq i \leq j \leq m, \forall \, 0 \leq q \leq \ell_j, k_{i,0} < k_{j,q}. 
\end{equation}
\end{lm}
This condition can be read as ``the first term of a tuple is smaller than the following terms''.

\begin{proof} 
Each direction is easy to check.
\end{proof}

If $B_i$ and $B_{i+1}$ are two consecutives blocks of a split tuple, we denote by $B_iB_{i+1}$ their concatenation.

\begin{lm} 
Let $(B_0 \mid B_1 \mid \ldots \mid B_m)$ be a split tuple, and $0 \leq i <m$.

Then $(B_0 \mid B_1 \mid \ldots \mid B_iB_{i+1} \mid \ldots \mid B_m)$
is also a split tuple. 
\end{lm}

\begin{proof} 
The condition (3) in the above lemma for $(B_0 \mid B_1 \mid \ldots \mid B_m)$ directly implies the condition for $(B_0 \mid B_1 \mid \ldots \mid B_iB_{i+1} \mid \ldots \mid B_m)$.
\end{proof}

We now use these tuples and split tuples to describe the bases $\mathcal B ^{\GLie}(n,m)$ and $\mathcal B ^{\GLie}(n,\N)$.

\begin{lm} \label{bases}
\begin{itemize}
 \item The set $\mathcal B ^{\GLie}(n,m)$ is in bijection with $STuple(n,m)$. In particular, $\mathcal B ^{\GLie}(n,0)$ is in bijection with $Tuple(n)$
 \item The set $\mathcal B ^{\GLie}(n,\N)$ is in bijection with $STuple(n, \N)$.
\end{itemize}
\end{lm}

\begin{proof}

Let $(\alpha, \sigma_0, \ldots, \sigma_m)$ be an element in $\mathcal B ^{\GLie}(n,m)$ where 
$\alpha \in \Sh([n],[m])$ and $\forall i \in \{0,\ldots,m\}$, $\sigma_i \in \mathcal B ^{Lie}(\alpha^{-1}(i))\simeq \mathfrak{S}_{\mid  \alpha^{-1}(i)\mid -1}$.
 We consider for all $i \in \{0, \ldots, m\}$ the ordered set $\alpha^{-1}(i)=\{k_{i,0}<k_{i,1}< \ldots<k_{i,\mid  \alpha^{-1}(i)\mid} \}$. 
We associate to $(\alpha, \sigma_0, \ldots, \sigma_m)$ the ordered partition of $[n]$ into $m+1$ tuples $(B_0 \mid B_1 \mid \ldots \mid B_m)$
 where for all $i \in \{0, \ldots, m\}$, $B_i=\{k_{i,0}, k_{i, \sigma_i(1)}, \ldots, k_{i, \sigma_i(\mid  \alpha^{-1}(i)\mid})\}$. 
We have $(B_0 \mid B_1 \mid \ldots \mid B_m) \in STuple(n,m)$ since condition $(1)$ in Definition \ref{STuple} 
corresponds to the shuffling property of $\alpha \in \Sh([n],[m])$ and condition $(2)$ is obviously satisfied. So we define the map
$$\Psi: \mathcal B ^{\GLie}(n,m) \to STuple(n,m)$$
 by:
$$\Psi(\alpha, \sigma_0, \ldots, \sigma_m)=(B_0 \mid B_1 \mid \ldots \mid B_m).$$
This map $\Psi$ is a bijection whose inverse is given by:
$$\Psi^{-1}(B_0 \mid B_1 \mid \ldots \mid B_m)=\Psi^{-1}(k_{0,0}, \ldots, k_{0, l_0} \mid \ldots \mid k_{n,0}, \ldots, k_{n, l_n})=(\alpha, \sigma_0, \ldots, \sigma_m)$$
where for all $i \in \{0, \ldots, m\}$, $\alpha(k_{i,0})= \ldots= \alpha(k_{i, l_i})=i$ and $\sigma_i$ is the permutation 
such that $k_{i,0}<k_{i,\sigma_i^{-1}(1)}<k_{i,\sigma_i^{-1}(2)}< \ldots< k_{i,\sigma_i^{-1}(l_i)}.$
The map $\alpha$ is a shuffle map by condition $(1)$ in Definition \ref{STuple}
\end{proof}

\begin{exple}
 The following basis element of $\GLie([6], [2])$:
$$ 
\vcenter{\xymatrix@M=0pt@H=0pt@R=13pt@C=13pt@!0{
0\ar@{-}[dr]& & 1\ar@{-}[dl] & \\
& \ar@{-}[dr]& & 5\ar@{-}[dl]\\
& & \ar@{-}[d] \\
& & 0&}}
\
\vcenter{\xymatrix@M=0pt@H=0pt@R=13pt@C=13pt@!0{
2\ar@{-}[ddd]\\
\\
\\
1&}}
\
\vcenter{\xymatrix@M=0pt@H=0pt@R=13pt@C=13pt@!0{
3\ar@{-}[dr]& & 6\ar@{-}[dl] & \\
& \ar@{-}[dr]& & 4\ar@{-}[dl]\\
& & \ar@{-}[d] \\
& & 2&}}$$
corresponds to the $2$-split $7$-tuple  in $STuple(6,2)$

$$
(0,1,5 \mid 2 \mid 3,6,4)
$$
by the bijection in Lemma \ref{bases}.
\end{exple}

In the rest of the paper, we identify the set  $\mathcal B ^{\GLie}(n,m)$ with $STuple(n,m)$.

\subsubsection{Ordering and filtration }

We now define an ordering on the set $\mathcal B^{\GLie}(n,\N)$ and use it to endow $\bigoplus_m \GLie([n],[m])$ with a filtration
(note that this direct sum is actually finite).

\begin{obs}
Forgetting the vertical bars of an element of $STuple(n, \N)$ gives an element of $Tuple(n)$. This defines a surjective map $p: STuple(n, \N) \to Tuple(n)$. 
\end{obs}

Using the bijection between $\mathcal B ^{\GLie}(n,\N)$ and $STuple(n, \N)$, we define $p:  \mathcal B ^{\GLie}(n,\N) \to Tuple(n)$, which is again a surjection.

We order the set $Tuple(n)$ with the usual total lexicographical ordering.

We now define an ascending filtration on $\bigoplus_m \GLie([n],[m])$.

\begin{lm}
Let $u$ be an element in $Tuple(n)$. The $\kk$-modules:
$$F_u:=\displaystyle{\bigoplus_{b \in \mathcal B ^\GLie(n, \N), p(b) \geq u} \kk b}$$
define an ascending filtration of the $\kk$-module $\bigoplus_m \GLie([n],[m])$.

The associated graded object is $\displaystyle{gr\Big(\bigoplus_m \GLie([n],[m])\Big) = \bigoplus_{u \in Tuple(n)} gr_u}$
where
$gr_u$ is $\displaystyle{\bigoplus_{b \in \mathcal B ^\GLie(n, \N), p(b) = u} \kk b}$.
\end{lm}

Recall that $C^{Leib}_m(P_n)= \GLie([n],[m])$ as $k$-modules. Thus we obtain the following proposition:

\begin{prop} \label{filtration}
We obtain a filtration on $C^{Leib}_*(P_n)$ (seen as a $\kk$-module) indexed by elements $u$ in $Tuple(n)$.
\end{prop}

We have to show that this filtration is compatible with the differential. This is the point of the next subsection.

\subsection{Compatibility of the filtration with the differential}

We first need to rewrite in $\mathcal B^{Lie}$ a product of two basis elements.
This allows us to understand the image by $P_n(d_{i,j})=(d_{i,j})_*$ (where $-_*$ denotes the postcomposition) of an element in  $\mathcal B^{\GLie}(n,\N)$. 
We actually do not need all the terms of the image, but only its leading term relatively to the ordering defined in the previous section.

\subsubsection{Products of elements in the basis of $Lie$}
We need to rewrite in $\mathcal B^{Lie}$ such a composite of two basis elements:
$$
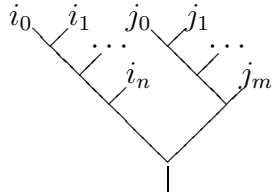
\xymatrix@M=0pt@H=0pt@W=0pt@R=11pt@C=11pt@!0{
i_0\ar@{-}[dr]& & i_1\ar@{-}[dl] & & j_0\ar@{-}[dr] & & j_1\ar@{-}[dl] & & \\
& \ar@{-}[dr] & & \cdots \ar@{-}[dl] & & \ar@{-}[dr] & & \cdots\ar@{-}[dl] & \\
& & \ar@{-}[dr] & & i_n\ar@{-}[dl] & & \ar@{-}[dr] & & j_m\ar@{-}[dl] \\
& & & \ar@{-}[ddrr] & & & & \ar@{-}[ddll]  \\
& & & & & & & \\
& & & & &\ar@{-}[d] & &\\
& & & & & &
}$$
where $i_0<j_0$, $i_0<i_k$ for all $1\leq k \leq n$ and $j_0 < j_k$ for all $1 \leq k \leq m$.

\begin{lm}\label{crochet} 
 The decomposition of this element in the basis $\mathcal B^{Lie}$ is the following sum:
$$\sum_{k=0}^m \
\sum_{S \subset \{1, \ldots, m \}, \mid S \mid =k} (-1)^k
\vcenter{\xymatrix@M=0pt@H=0pt@R=11pt@C=11pt@!0{
i_0\ar@{-}[dr]& & i_1\ar@{-}[dl] \\
& \ar@{-}[dr] & & \cdots\ar@{-}[dl]  \\
& & \ar@{-}[dr] & & i_n\ar@{-}[dl] \\
& & & \ar@{-}[dr] & & \cdots \ar@{-}[dl] \\
& & & & \ar@{-}[dr] & & j_0\ar@{-}[dl] \\
& & & & & \ar@{-}[dr] & & \cdots\ar@{-}[dl] \\
& & & & & & \ar@{-}[d] \\
& & & & & & & & & & 
}}$$
where: \\
- the inputs between $i_0$ and $i_n$ are in the same order as before,\\
- $j_0$ is in the $(n+2+k)$-th position, \\
- the inputs between $i_n$ and $j_0$ are labelled by the $j_\ell$ for $\ell \in S$ (with $\ell$ decreasing from top to bottom),\\
- the inputs below $j_0$ are labelled by the $j_\ell$ for $\ell \notin S$ (with $\ell$ increasing from top to bottom).
\end{lm}

Let us denote by $b_k^S$ the element of the basis appearing for the indices $k$ of the first sum and $S$ of the second sum.

\begin{proof}[Sketch of proof]
 We proceed by induction on $m$.

For $m=0$, there is nothing to prove.

For $m=1$, we apply the Jacobi relation. 

To prove the formula for $m+1$ if $m \geq 1$, the idea is to apply the formula for $m$ to the tree where 
$ \vcenter{\xymatrix@M=0pt@H=0pt@R=7pt@C=10pt@!0{ j_0\ar@{-}[dr]& & j_1\ar@{-}[dl] \\ &\ar@{-}[d] &\\&&}}$ 
has been replaced by $j_0$. Then we graft 
$ \vcenter{\xymatrix@M=0pt@H=0pt@R=7pt@C=10pt@!0{ j_0\ar@{-}[dr]& & j_1\ar@{-}[dl] \\ &\ar@{-}[d] &\\&&}}$ 
back in every term where $j_0$ appears. Finally we use the Jacobi relation again, and obtain the desired formula for $m+1$. We refer the reader to the proof of  \cite[Theorem 5.1]{Reutenauer} for details.
\end{proof}

\begin{rem}
In \cite[Theorem 5.1]{Reutenauer} the description of the basis $\mathcal B^{Lie}$ corresponds to the fact that Lyndon words form a Hall set and the previous lemma corresponds to the triangularity property of the corresponding Hall basis. Notice that a careful study of the proof of  \cite[Theorem 5.1]{Reutenauer}  gives the more precise formula given in the statement of Lemma \ref{crochet}.
\end{rem}

\begin{rem}
This lemma moreover implies that the basis satisfies a property very similar to the Poincar\'e-Birkhoff-Witt property of associative algebras.
It is then possible to show that the category $\GLie$ is actually a Koszul category.
\end{rem}

\begin{exple} \label{exampled01}
 $$\vcenter{\xymatrix@M=0pt@H=0pt@W=0pt@R=13pt@C=13pt@!0{
i_0\ar@{-}[dr]& & i_1\ar@{-}[dl] & & j_0\ar@{-}[dr] & & j_1\ar@{-}[dl] & & \\
& \ar@{-}[dr] & &  & & \ar@{-}[dr] & & j_2\ar@{-}[dl] & \\
& & \ar@{-}[dr] & &  & & \ar@{-}[dr] & & j_3\ar@{-}[dl] \\
& & & \ar@{-}[ddrr] & & & & \ar@{-}[ddll]  \\
& & & & & & & \\
& & & & &\ar@{-}[d] & &\\
& & & & & &
}}
=$$

$$
\vcenter{\xymatrix@M=0pt@H=0pt@R=13pt@C=13pt@!0{
i_0\ar@{-}[dr]& & i_1\ar@{-}[dl] \\
& \ar@{-}[dr] & & j_0\ar@{-}[dl]  \\
& & \ar@{-}[dr] & & j_1\ar@{-}[dl] \\
& & & \ar@{-}[dr] & & j_2 \ar@{-}[dl] \\
& & & & \ar@{-}[dr] & & j_3\ar@{-}[dl] \\
& & & & & \ar@{-}[d] \\
& & & & & &  
}}
-
\vcenter{\xymatrix@M=0pt@H=0pt@R=13pt@C=13pt@!0{
i_0\ar@{-}[dr]& & i_1\ar@{-}[dl] \\
& \ar@{-}[dr] & & j_1\ar@{-}[dl]  \\
& & \ar@{-}[dr] & & j_0\ar@{-}[dl] \\
& & & \ar@{-}[dr] & & j_2 \ar@{-}[dl] \\
& & & & \ar@{-}[dr] & & j_3\ar@{-}[dl] \\
& & & & & \ar@{-}[d] \\
& & & & & &  
}}
-
\vcenter{\xymatrix@M=0pt@H=0pt@R=13pt@C=13pt@!0{
i_0\ar@{-}[dr]& & i_1\ar@{-}[dl] \\
& \ar@{-}[dr] & & j_2\ar@{-}[dl]  \\
& & \ar@{-}[dr] & & j_0\ar@{-}[dl] \\
& & & \ar@{-}[dr] & & j_1 \ar@{-}[dl] \\
& & & & \ar@{-}[dr] & & j_3\ar@{-}[dl] \\
& & & & & \ar@{-}[d] \\
& & & & & & 
}}
-
\vcenter{\xymatrix@M=0pt@H=0pt@R=13pt@C=13pt@!0{
i_0\ar@{-}[dr]& & i_1\ar@{-}[dl] \\
& \ar@{-}[dr] & & j_3\ar@{-}[dl]  \\
& & \ar@{-}[dr] & & j_0\ar@{-}[dl] \\
& & & \ar@{-}[dr] & & j_1 \ar@{-}[dl] \\
& & & & \ar@{-}[dr] & & j_2\ar@{-}[dl] \\
& & & & & \ar@{-}[d] \\
& & & & & &  
}}
+$$

$$ 
\vcenter{\xymatrix@M=0pt@H=0pt@R=13pt@C=13pt@!0{
i_0\ar@{-}[dr]& & i_1\ar@{-}[dl] \\
& \ar@{-}[dr] & & j_2\ar@{-}[dl]  \\
& & \ar@{-}[dr] & & j_1\ar@{-}[dl] \\
& & & \ar@{-}[dr] & & j_0 \ar@{-}[dl] \\
& & & & \ar@{-}[dr] & & j_3\ar@{-}[dl] \\
& & & & & \ar@{-}[d] \\
& & & & & &  
}}
+
\vcenter{\xymatrix@M=0pt@H=0pt@R=13pt@C=13pt@!0{
i_0\ar@{-}[dr]& & i_1\ar@{-}[dl] \\
& \ar@{-}[dr] & & j_3\ar@{-}[dl]  \\
& & \ar@{-}[dr] & & j_1\ar@{-}[dl] \\
& & & \ar@{-}[dr] & & j_0 \ar@{-}[dl] \\
& & & & \ar@{-}[dr] & & j_2\ar@{-}[dl] \\
& & & & & \ar@{-}[d] \\
& & & & & &  
}}
+
\vcenter{\xymatrix@M=0pt@H=0pt@R=13pt@C=13pt@!0{
i_0\ar@{-}[dr]& & i_1\ar@{-}[dl] \\
& \ar@{-}[dr] & & j_3\ar@{-}[dl]  \\
& & \ar@{-}[dr] & & j_2\ar@{-}[dl] \\
& & & \ar@{-}[dr] & & j_0 \ar@{-}[dl] \\
& & & & \ar@{-}[dr] & & j_1\ar@{-}[dl] \\
& & & & & \ar@{-}[d] \\
& & & & & &  
}}
-
\vcenter{\xymatrix@M=0pt@H=0pt@R=13pt@C=13pt@!0{
i_0\ar@{-}[dr]& & i_1\ar@{-}[dl] \\
& \ar@{-}[dr] & & j_3\ar@{-}[dl]  \\
& & \ar@{-}[dr] & & j_2\ar@{-}[dl] \\
& & & \ar@{-}[dr] & & j_1 \ar@{-}[dl] \\
& & & & \ar@{-}[dr] & & j_0\ar@{-}[dl] \\
& & & & & \ar@{-}[d] \\
& & & & & &  
}}$$
$$=b_0^{\emptyset}-b_1^{\{1\}}-b_1^{\{2\}}-b_1^{\{3\}}+b_2^{\{1,2\}}+b_2^{\{1,3\}}+b_2^{\{2,3\}}-b_3^{\{1,2,3\}}.$$
\end{exple}

Using our notations for the basis of $\GLie([n],[m])$, we can notice that the previous example can be seen as 
applying $(d_{0,1})_*$ to
$(i_0, i_1 | j_0, j_1, j_2, j_3) \in \GLie([5],[1])$, which gives an element in $\GLie([5],[0])$.
And the sum of eight terms we obtain is the decomposition of  $(d_{0,1})_*(i_0, i_1 | j_0, j_1, j_2, j_3)$ in the basis of $\GLie([5],[0])$.
The sum begins with : 
$$(i_0, i_1, j_0, j_1, j_2, j_3) - (i_0, i_1, j_1, j_0, j_2, j_3) - \ldots $$
We can notice that the first term of the sum is in the same stage of the filtration as the element we were starting with, 
while the other 7 terms are in a higher stage of the filtration.

\vspace{1cm}

Let us make some additional observations about Lemma~\ref{crochet} in the general case:
\begin{enumerate}
 \item For $k=0$, the sum on subsets $S$ is composed of a single term for $S=\emptyset$. 
       This term $b_0^{\emptyset}$ in the basis is identified with the tuple $(i_0, i_1, \ldots, i_n, j_0, j_1, \ldots, j_n)$, and has a $+1$ coefficient.
       This means that, if we denote $B_i=(i_0, \ldots, i_n)$ and $B_j=(j_0, \ldots, j_m)$, we obtain the equality $b_0^{\emptyset}=B_iB_j$. 
       Notice that $p(b_0^{\emptyset})=p(B_iB_j)=p(B_i \mid B_j)$.
 \item For $k \geq 1$, notice that all the terms $b_k^S$ appearing in the sum have a $(n+2)$th label equal to $j_l$ for $l>0$.
       As $j_0$ is the $(n+2)$th label of $p(B_i \mid B_j)$ and $j_l > j_0$, using the lexicographical ordering in $Tuple(n)$, 
       we obtain the inequality  $p(b_k^S) > p(B_i \mid B_j)$.
\end{enumerate}

\subsubsection{Compatibility of the filtration with the differential}

We now want to see how the observation from the previous example can be generalized to $(d_{i,j})_*(b)$ where $b$ is a basis element $\GLie([n],[m])$.

\begin{lm}\label{calculdiff}
Let $b=(B_0 \mid \ldots \mid B_m)$ be in $\mathcal B^\GLie (n,m)$ where $m>0$. Let $u$ denote $p(b)$.
\begin{enumerate}
\item [(i)]
If $0 \leq i <m$, then $(d_{i,i+1})_*(b)$ is in $F_u$. Moreover, in $gr_u$,  we obtain $(d_{i,i+1})_*(b)= (B_0 \mid \ldots \mid B_iB_{i+1} \mid \ldots \mid B_m)$.

\item [(ii)]
If $0 \leq i <m$ and $i+2 \leq j \leq m$, then $d_{i,j}(b)$ is in $F_u$. Moreover, in $gr_u$, we obtain $d_{i,j}(b)= 0$.
\end{enumerate}

\end{lm}

That means that $(d_{i,i+1})_*(b)$ has one term (with a plus sign) in the same stage of the filtration as $b$, and other terms in higher stages.
And $(d_{i,j})_*(b)$ for $j>i+1$ has only terms in stages of the filtration higher than the stage containing $b$, and thus vanishes in the graded object.

\begin{proof}

By Lemma \ref{crochet} we see that $(d_{i,i+1})_*(B_0 \mid \ldots \mid B_m)$ is 
a sum of the term $(B_0 \mid \ldots \mid B_iB_{i+1} \mid \ldots \mid B_m)$ 
with the terms $\pm (B_0 \mid \ldots \mid b_k^S \mid \ldots \mid B_m)$ where $k>0$.

Using the projection to $Tuple(n)$ and the lexicographical ordering there, we notice
$u=p(B_0 \mid \ldots \mid B_iB_{i+1} \mid \ldots \mid B_m)$ and $ u < p(B_0 \mid \ldots \mid b_k^S \mid \ldots \mid B_m)$.

Thus we obtain the part (i) of the lemma.

For the second part, first notice that $(d_{i,j})_*(B_0 \mid \ldots \mid B_m)$ is a sum of terms whose smallest one is 
$(B_0 \mid \ldots \mid B_iB_j \mid \ldots \mid \widehat{B_j} \mid \ldots \mid B_m)$.
But $p(B_0 \mid \ldots \mid B_iB_j \mid \ldots \mid \widehat{B_j} \mid \ldots \mid B_m)$ is larger than $p(B_0 \mid \ldots \mid B_m)$ 
because $k_{i+1,0} < k_{j,0}$ as $i+1<j$ by condition $(1)$ in Definition \ref{STuple}.

Thus all terms appearing in $d_{i,j}(B_0 \mid \ldots \mid B_m)$ are in $F_u$ and vanish in $gr_u$.
This concludes the proof of (ii).
\end{proof}

\begin{exple}
An example of the behaviour of $(d_{i,i+1})_*$ was already given in the observation after Example \ref{exampled01}.

Let us now consider an example of the second case: $(d_{0,2})_*$ applied to the following element of the basis of $\GLie([6], [2])$:
$$ b=
\vcenter{\xymatrix@M=0pt@H=0pt@R=13pt@C=13pt@!0{
0\ar@{-}[dr]& & 1\ar@{-}[dl] & \\
& \ar@{-}[dr]& & 5\ar@{-}[dl]\\
& & \ar@{-}[d] \\
& & 0&}}
\
\vcenter{\xymatrix@M=0pt@H=0pt@R=13pt@C=13pt@!0{
2\ar@{-}[ddd]\\
 \\
 \\
1&}}
\
\vcenter{\xymatrix@M=0pt@H=0pt@R=13pt@C=13pt@!0{
3\ar@{-}[dr]& & 6\ar@{-}[dl] & \\
& \ar@{-}[dr]& & 4\ar@{-}[dl]\\
& & \ar@{-}[d] \\
& & 2&}}.$$

The computation gives a sum of four terms of the form
$$\vcenter{\xymatrix@M=0pt@H=0pt@R=13pt@C=13pt@!0{
0\ar@{-}[dr]& & 1\ar@{-}[dl] &&&&& 2\ar@{-}[dddddd]\\
& \ar@{-}[dr] & & 5\ar@{-}[dl]  &&&&\\
& & \ar@{-}[dr] & & \sigma(3)\ar@{-}[dl] &&&\\
& & & \ar@{-}[dr] & & \sigma(6) \ar@{-}[dl] &&\\
& & & & \ar@{-}[dr] & & \sigma(4)\ar@{-}[dl] & \\
& & & & & \ar@{-}[d] & & \\
& & & &  & 0 & &1
}}$$
where $\sigma$ is a bijection of the set $\{ 3, 4, 6 \}$.

As split tuples, the equality reads
$$(d_{0,2})_*(0,1,5 \mid 2 \mid 3,6,4) = $$
$$(0,1,5,3,6,4 \mid 2) - (0,1,5,6,3,4 \mid 2) - (0,1,5,4,3,6 \mid 2) + (0,1,5,6,4,3 \mid 2).$$

Clearly, when projecting on Tuples(6), the four terms of the right hand side are larger than 
the term of the left hand side.
Thus in the associated graded object, the equality reduces to 
$$(d_{0,2})_*(0,1,5 \mid 2 \mid 3,6,4) \cong 0.$$
\end{exple}

\begin{prop} \label{compatible}
The differential of $C^{Leib}_*(P_n)$  is compatible with the filtration.
\end{prop}

\begin{proof}
 This is a direct consequence of the previous lemma, using that the differential in  $C^{Leib}_*(P_n)$ is $d= \displaystyle{\sum_{0\leq i<j \leq n} (-1)^j (d_{i,j})_*}$.
\end{proof}

\subsection{Acyclicity of the associated graded complex}

We now identify the associated graded complex with a sum of acyclic ones.

First, let us describe the differential in the associated graded complex. From Lemma \ref{calculdiff}, we obtain
$$d_{gr}(B_0 \mid \ldots \mid B_m)= \sum_{0\leq i <m} (-1)^{i+1} (B_0 \mid \ldots \mid B_iB_{i+1} \mid \ldots \mid B_m).$$

For a fixed $u=(0, k_1, \ldots, k_n) \in Tuple(n)$, we want to describe the elements $b\in STuple(n)$ such that $p(b)=u$.
We use the definition of a split tuple with the condition (3) from Lemma \ref{condition3}, which characterizes the first terms of a block.

\begin{defi}
We call an admissible cut of the tuple $u=(0, k_1, \ldots, k_n)$ an index $i$ with $1 \leq i \leq n$ satisfying the condition
$$\forall \, 1 \leq j \leq m, j>i \Rightarrow k_j > k_i.$$
We denote by $Cuts(u)$ the set of admissible cuts of $u$.
\end{defi}

Notice that $Cuts(u)$ is non-empty as $n$ always belongs to $Cuts(u)$ (the condition of the definition is empty for this index).

\begin{exple}
 Let $u$ be $(0,2,1,5,3,4)$. The set $Cuts(u)$ is $\{2,4,5\}$. 
These indices $2$, $4$ and $5$ are exactly the ranks where it is possible to split the tuple.
For instance, $(0,2 \mid 1,5 \mid 3 \mid 4)$ is the split tuple obtained by cutting as most as possible.
\end{exple}

\begin{lm}
The set of $b\in STuple(n)$ such that $p(b)=u=(0, k_1, \ldots, k_n)$ is in bijection with the power set $\mathscr P (Cuts(u))$, via 
$$(B_0 \mid \ldots \mid B_m) \mapsto \{ |B_0|, |B_0| + |B_1|, \ldots, |B_0| + \ldots + |B_{m-1}| \}$$

The inverse bijection is given by 
$$\{i_1, \ldots, i_m\} \mapsto (0, k_1, \ldots, k_{i_1 -1} \mid k_{i_1}, k_{i_1+1}, \ldots  \ldots \mid k_{i_2} \ldots \mid \ldots \mid k_{i_m}, k_{i_m+1}, \ldots ).$$
\end{lm}

Recall that for any non-empty finite set $E=\{e_0, \ldots, e_n\}$, 
there exists an associated chain complex $C_*(E)$ whose basis is the power set $\mathscr P (E)$, a subset of cardinality $k$ giving a generator in degree $k$.
The differential $d: C_{l+1} \to C_l$ is given by $d(\{e_{i_0}, \ldots e_{i_\ell} \}) = \displaystyle{ \sum_{j=0}^\ell (-1)^j \{e_{i_0}, \ldots, \widehat{e_{i_j}}, \ldots, e_{i_\ell} \} }$.
This complex $C_*(E)$ is acyclic (for instance it can be seen as a complex associated to a simplex).

\begin{lm}
 There is an isomorphism of complexes between $(gr_u,d_{gr})$ and $(C_*(Cuts(u)),d)$.
\end{lm}

\begin{proof}
Recall that in the graded object, the part $d_{i,i+1}$ of the differential is the concatenation of the blocks $B_i$ and $B_{i+1}$, that is the removal of the $(i+1)$th cut.
The part $d_{i,i+1}$ comes with the sign $(-1)^{i+1}$, and in the differential of the chain complex $(C_*(Cuts(u)),d)$, removing the $(i+1)$th cut comes with the same sign.
\end{proof}

This isomorphim and the acyclicity of the complex $\big(C_*(Cuts(u)),d\big)$ imply the following lemma:

\begin{lm}
 The complex $(gr_u, d_{gr})$ is acyclic for  $u \in Tuple(n)$.
\end{lm}

We can now finish the proof of Proposition \ref{Leib-projectifs}.


\begin{proof}[Proof of Proposition \ref{Leib-projectifs}]
The filtration of $C_*^{Leib}(P_n)$ is bounded since $Tuple(n)$ is a finite set.
So, by the classical convergence theorem of the spectral sequence associated to a filtration (see \cite[Theorem 5.5.1]{Weibel}), this spectral sequence converges to $H_*^{Leib}(P_n)$.
The previous lemma implies that 
$\displaystyle{gr C^{Leib}_*(P_n) = \bigoplus_{u \in Tuple(n)} gr_u} $ is a sum of acyclic complexes, and therefore is also acyclic. We deduce that $H_*^{Leib}(P_n)=0$
\end{proof}



\bibliographystyle{amsplain}
\bibliography{biblio-Gamma-Lie}

\end{document}